\documentclass[12pt]{amsart}
\usepackage{graphicx}
\usepackage[active]{srcltx}
\usepackage{amsthm,mathrsfs}
\usepackage{a4wide}
\usepackage{amsfonts}
\usepackage{amsmath}
\usepackage{amssymb}
\usepackage{dsfont}
\newtheorem{lemma}{Lemma}
\newtheorem{theorem}{Theorem}

\usepackage{float}
\usepackage{url}

\newcommand {\N} {\mathbb{N}}

\newcommand {\ve} {\varepsilon}

\makeatletter\def\blfootnote{\xdef\@thefnmark{}\@footnotetext}\makeatother
\allowdisplaybreaks
\title[Lower bounds for the Riemann zeta function]{\bf Lower bounds for the maximum of the Riemann zeta function along vertical lines}

\author{Christoph Aistleitner} 
\address{Department of Mathematics, Graduate School of Science, Kobe University, Kobe 657-8501, Japan}
\email{aistleitner@math.tugraz.at}

\thanks{The author is supported by a Schr\"odinger scholarship of the Austrian Research
Foundation (FWF)}


\begin{document}

\begin{abstract}
Let $\alpha \in (1/2,1)$ be fixed. We prove that 
$$
\max_{0 \leq t \leq T} |\zeta(\alpha+it)| \geq \exp\left(\frac{c_\alpha (\log T)^{1-\alpha}}{(\log \log T)^\alpha}\right)
$$
for all sufficiently large $T$, where we can choose $c_\alpha = 0.18 (2\alpha-1)^{1-\alpha}$. The same result has already been obtained by Montgomery, with a smaller value for $c_\alpha$. However, our proof, which uses a modified version of Soundararajan's ``resonance method'' together with ideas of Hilberdink, is completely different from Montgomery's. This new proof also allows us to obtain lower bounds for the measure of those $t \in [0,T]$ for which $|\zeta(\alpha+it)|$ is of the order mentioned above.
\end{abstract}

\date{}
\maketitle

\section{Introduction and statement of results}

The Lindel\"of Hypothesis asserts that for every $\ve>0$ we have
$$
|\zeta(1/2+it)| = \mathcal{O} \left(t^\ve\right) \qquad \textrm{as $t \to \infty$.} 
$$
Assuming the Riemann Hypothesis, one can even show that
$$
|\zeta(1/2+it)| = \mathcal{O} \left( \exp \left(\frac{c \log t}{\log \log t} \right) \right)  \qquad \textrm{as $t \to \infty$} 
$$
(see~\cite{chand}; here, and in the sequel, we write $\exp(x)$ for $e^x$, and we write $c$ and $c_\alpha$ for generic positive constants). On the other hand, it is known that
\begin{equation} \label{montlow}
\max_{0 \leq t \leq T} \left| \zeta(1/2+it) \right| = \Omega \left( \exp \left(\frac{c \sqrt{\log T}}{\sqrt{\log \log T}} \right) \right) \qquad \textrm{as $t \to \infty$}; 
\end{equation}
this was proved by Montgomery~\cite[Theorem 2]{mont} under the assumption of the Riemann Hypothesis, and by Balasubramanian and Ramachandra~\cite{bala} unconditionally. Soundararajan~\cite{sound} gave an alternative proof; in his paper he introduced the \emph{resonance method}, which is also the main tool in the proof in the present paper.\\

Concerning the order of $|\zeta(\alpha+it)|$ for $\alpha \in (1/2,1)$, Montgomery~\cite[Theorem 1]{mont} proved that
\begin{equation} \label{mont}
\max_{0 \leq t \leq T} |\zeta(\alpha+it)| = \Omega \left( \exp \left( \frac{c_\alpha (\log T)^{1-\alpha}}{(\log \log T)^\alpha}\right) \right) \qquad \textrm{as $t \to \infty$},
\end{equation}
where we can take $c_\alpha = (\alpha-1/2)^{1/2}/20$ (he also showed that assuming the Riemann Hypothesis we may take $c_\alpha = 1/20$ for all $\alpha \in (1/2,1)$). Based on probabilistic arguments, Montgomery conjectured that the lower bound in~\eqref{mont} is optimal, except for the precise value of $c_\alpha$.\footnote{The situation for $\alpha=1/2$ seems to be more dubious. Montgomery wrote that for $\alpha=1/2$ the ``situation is more delicate'', but he ``still suggests'' that~\eqref{montlow} could be optimal. On the other hand, Farmer, Gonek, and Hughes~\cite{farmer} conjectured that $\max_{t\in [0,T]} \left| \zeta(1/2+it)\right| = \exp \left( (1+o(1)) \sqrt{\frac{1}{2} \log T \log \log T}\right)$. See also~\cite{kotnik}.} On the other hand, conditional under the Riemann Hypothesis we have the upper bound
\begin{equation} \label{riemupp}
|\zeta(\alpha+it)| = \mathcal{O} \left( \exp \left(\frac{c_\alpha (\log t)^{2-2\alpha}}{\log \log t} \right) \right)
\end{equation}
(see~\cite{lamz,titch}). It should be noted that the case $\alpha=1$ is quite different from the case $\alpha \in (1/2,1)$; it is known that $\max_{0 \leq t \leq T} |\zeta(1+it)|$ is of order roughly $\log \log T$. For details, see~\cite{gs}.\\

Hilberdink~\cite{hilber} used a version of Soundararajan's method to obtain lower bounds for $\max_{0 \leq t \leq T} |\zeta(\alpha+it)|$ in the case $\alpha \in (1/2,1)$; however, he could only prove a weaker result than Montgomery, with the term $(\log \log T)^\alpha$ replaced by $\log \log T$ in~\eqref{mont}. In the course of his proof, Hilberdink established an intimate connection between the problem concerning the maximum of the Riemann zeta function and certain sums involving greatest common divisors (GCD sums), which forms a cornerstone of the proof in the present paper. These GCD sums have already appeared in the context of several other problems, from Diophantine approximation to linear algebra and convergence problems for series of dilated functions; see for example~\cite{abs,gal,hed}. A detailed exposition of the connection of these GCD sums with other areas of mathematics can be found in~\cite{abs2}. A recent manuscript of Lewko and Radziwi{\l}{\l} intensifies the connection between GCD sums and 
properties of the Riemann zeta function, see~~\cite{lewko}.\\

The main result of the present paper is the following theorem.

\begin{theorem} \label{th1}
Let $\alpha \in (1/2,1)$ be fixed. Then we have
$$
\max_{0 \leq t \leq T} |\zeta(\alpha+it)| \geq \exp\left(\frac{c_\alpha (\log T)^{1- \alpha}}{(\log \log T)^\alpha}\right)
$$
for all sufficiently large $T$, where we can choose
$$
c_\alpha = 0.18 (2\alpha-1)^{1-\alpha}.
$$
\end{theorem}

This results recaptures Montgomery's lower bound, with a better value for the constant in the exponential term. However, they key point of this paper is not so much to improve the constant in Montgomery's result, but rather to present an alternative proof, which overcomes several crucial issues which were unresolved in earlier papers following a similar approach. Our proof also allows us to obtain lower bounds for the measure of those $t$ for which $|\zeta(\alpha+it)|$ is of order $\exp\left(\frac{\tau (\log T)^{1- \alpha}}{(\log \log T)^\alpha}\right)$ for ``small'' values of $\tau$. As far as I know, these are the first such bounds for values of $|\zeta(\alpha+it)|$ which are so close to the conjectured maximal order. Earlier results, such as the very precise ones of Lamzouri~\cite{lamz}, could only give lower bounds for the measure of $\{t \in [0,T]:~|\zeta(\alpha+it)| \geq \kappa\}$ for smaller values of $\kappa$; in the case of~\cite{lamz}, $\kappa$ is restricted to $1 \leq \kappa \leq c_\alpha (\log T)
^{1-\alpha} / \log \log T$.

\begin{theorem} \label{th2}
Assume that $\alpha \in (1/2,1)$, and let $\tau$ be a number satisfying 
\begin{equation} \label{tau}
0 < \tau < \frac{\left(2a-1 \right)^{1-\alpha}}{6}.
\end{equation}
Furthermore, set 
\begin{equation} \label{falpha}
F_{\alpha,\tau} = \left\{ t \in [0,T]:~|\zeta(\alpha+it)| \geq \exp \left( \frac{\tau (\log T)^{1 - \alpha}}{(\log \log T)^\alpha} \right) \right\}.
\end{equation}
Then for all sufficiently large $T$ we have
$$
\textup{meas} (F_{\alpha,\tau}) \geq T^{2\alpha -1 - \beta},
$$
where 
$$
\beta=\beta(\tau) = (6 \tau)^{\frac{1}{1-\alpha}}.
$$
Note that here for all admissible choices of $\alpha$ and $\tau$ we have $2 \alpha - 1 - \beta > 0$.
\end{theorem}

The outline of the remaining part of this paper is as follows. In Section~\ref{sechilber} below, we recapitulate the main ingredients in the arguments of Hilberdink and Soundararajan, in order to expose the problems which arise in this approach. In Section~\ref{secproof} we show how to overcome some of the limitations of this argument, and prove Theorem~\ref{th1}. Section~\ref{seclemma} contains the proofs of several lemmas which are stated in Section~\ref{secproof}, but not proved there. In Section~\ref{secth2} we prove Theorem~\ref{th2}, and the final Section~\ref{secconc} contains some concluding remarks.\\

\section{The Hilberdink--Soundararajan argument revisited} \label{sechilber}
Let $\alpha \in [1/2,1)$ and $T$ be given (and assume that $T$ is ``large''). Let $\mu$ be a ``small'' constant, and let $M$ denote the integer nearest to $\mu (\log T)/(\log \log T)$. Let $(p_r)_{r \geq 1}$ denote the sequence of primes, sorted in increasing order. Let $\mathcal{B}=\{b_1, \dots, b_N\}$ denote the set of the $N=2^M$ positive integers of the form
\begin{equation} \label{prod}
\prod_{r=1}^M p_{r}^{\beta_r}, \qquad (\beta_1, \dots, \beta_M) \in \{0,1\}^M.
\end{equation}
In our argument, the GCD sum 
\begin{equation} \label{gcdsum}
\sum_{1 \leq k,\ell \leq N} \frac{(\gcd(b_k,b_\ell))^{2 \alpha}}{(b_k b_\ell)^\alpha}
\end{equation} 
will play an important role. By the specific structure of the numbers $b_1, \dots, b_N$, the sum~\eqref{gcdsum} can be easily calculated. For every fixed $k$ we have
$$
\sum_{1 \leq \ell \leq N} \frac{(\gcd(b_k,b_\ell))^{2 \alpha}}{(b_k b_\ell)^\alpha} = \prod_{r=1}^M \left( 1+p_r^{-\alpha} \right),
$$
and thus by the prime number theorem the value of~\eqref{gcdsum} is
\begin{equation} \label{gcdsum2}
N \prod_{r=1}^M \left( 1+p_r^{-\alpha} \right) \approx N \exp \left(\frac{c_\alpha (\log N)^{1- \alpha}}{(\log \log N)^\alpha}\right)
\end{equation}
(this observation already appears in~\cite{gal}).\\

We want to prove that $|\zeta(\alpha+it)|$ is large for some value of $t \in [0,T]$. To do so, the main idea of Soundararajan's \emph{resonance method} in~\cite{sound} is (roughly speaking) to construct a function $A(t)$ such that 
\begin{equation} \label{intsound}
\int_0^T \zeta(\alpha+it) |A(t)|^2 dt
\end{equation}
is ``large'', while at the same time $\int_0^T |A_N(t)|^2~dt$ is ``small''. This provides the desired lower bound, since $\max_{0 \leq t \leq T} |\zeta(\alpha+it)|$ must be at least as large as the quotient of the absolute value of the first integral divided by the second. Hilberdink's idea in~\cite{hilber} is to replace the integral in~\eqref{intsound} by
\begin{equation} \label{sounint}
\int_0^T |\zeta(\alpha+it)|^2 |A(t)|^2 dt,
\end{equation} 
and to link this integral to a GCD sum (see below).\\

In the following paragraph we will describe Hilberdink's version of the resonance method in a simplified form. We assume that $\alpha \in (1/2,1)$. By a classical approximation theorem for the Riemann zeta function we have
\begin{equation} \label{zetaapprox}
\zeta(\alpha+it) = \sum_{n \leq x} n^{-\alpha-it} + \frac{x^{1-\alpha-it}}{\alpha+it-1} + \mathcal{O} \left( x^{-\alpha} \right)
\end{equation}
uniformly for all $x$ satisfying $2 \pi x/C \geq |t|$, where $C$ is a given constant greater than 1 (see~\cite[Theorem 4.11]{titch}). We choose $x=t$ in~\eqref{zetaapprox}, and have
$$
\zeta(\alpha+it) = \sum_{n \leq t} n^{-\alpha-it} + \mathcal{O} \left(t^{-\alpha} \right)
$$
uniformly for $t \in [0,T]$, which implies that
\begin{eqnarray} 
\left|\zeta(\alpha+it)\right|^2 & = & \left| \sum_{n \leq t} \frac{1}{n^{\alpha+it}} \right|^2 + \mathcal{O} \left(t^{1-2\alpha} \right), \label{equ1}
\end{eqnarray}
also uniformly for $t \in [0,T]$. Let $A(t)$ the function given by
\begin{equation} \label{ant}
A(t) = \sum_{k=1}^N b_k^{it},
\end{equation}
where $b_1, \dots, b_N$ are the numbers defined at the beginning of this section. Note that $A(t)$ may also be written as a finite Euler product, since
\begin{equation} \label{ant2}
A(t) = \prod_{r=1}^M (1 + p_r^{it}).
\end{equation}
By~\eqref{equ1} we have
\begin{eqnarray}
& & \int_{0}^T |\zeta(\alpha+it) A(t)|^2 ~dt \nonumber\\
& = & \sum_{1 \leq k,\ell \leq N} \sum_{1 \leq m,n \leq T} \frac{1}{(mn)^\alpha} \int_{\max (m,n)}^T \left( \frac{m b_k}{n b_\ell} \right)^{it} dt + \mathcal{O} \left(N^2 T^{2-2\alpha} \right). \label{dobs}
\end{eqnarray}
The double sum in~\eqref{dobs} can be split into a sum over those indices $(k,\ell,m,n)$ for which $m b_k = n b_\ell$, and those indices for which this is not the case. Note that $m b_k = n b_\ell$ whenever $m=j b_\ell/\gcd(b_k,b_\ell)$ and $n=j b_k/\gcd(b_k,b_\ell)$ for some integer $j \geq 1$.\\

So far the particular choice of $M$ did not play any role. However, for the remaining part of the argument the size of $M$ is crucial, since it implies that for given $\ve>0$ by the prime number theorem we have
\begin{equation} \label{size}
b_k \leq \prod_{r=1}^M p_r \leq T^\ve,
\end{equation}
provided that $\mu$ was chosen sufficiently small. Thus for any $k,\ell$ there is at least one solution $(m,n),~1 \leq m,n \leq T^\ve,$ of the equation $m b_k = n b_\ell$, namely $m=b_\ell/\gcd(b_k,b_\ell)$ and $n=b_k/\gcd(b_k,b_\ell)$. Consequently we have
\begin{eqnarray}
\underbrace{\sum_{1 \leq k,\ell \leq N} \sum_{1 \leq m,n \leq T}}_{m b_k = n b_\ell} \frac{1}{(mn)^\alpha} \underbrace{\int_{\max(m,n)}^T \left( \frac{m b_k}{n b_\ell} \right)^{it} dt}_{\geq T - T^\ve} & \gg &
T \sum_{1 \leq k,\ell \leq N} \frac{(\gcd(b_k,b_\ell))^{2 \alpha}}{(b_k b_\ell)^\alpha} \nonumber\label{contria}\\
& \gg & NT \exp \left(c_\alpha \frac{(\log N)^{1- \alpha}}{(\log \log N)^\alpha}\right), \label{contri}
\end{eqnarray}
where we used~\eqref{gcdsum2}. The contribution of those indices $(k,\ell,m,n)$ for which $m b_k \neq n b_\ell$ is small; here it is crucial that by~\eqref{size} the quotients $m b_k/(n b_\ell)$ cannot be arbitrarily close to~1.  The error term on the right-hand side of~\eqref{dobs} is negligible, since $N = 2^M$ is sufficiently small due to our choice of $M$. Thus overall we have
\begin{equation} \label{l2nan}
\int_0^T |\zeta(\alpha+it) A(t)|^2 ~dt \gg  NT \exp \left(\frac{c_\alpha (\log N)^{1- \alpha}}{(\log \log N)^\alpha}\right).
\end{equation}

On the other hand, we have
\begin{eqnarray}
\int_0^T |A(t)|^2 ~dt & = & \int_0^T \sum_{1 \leq k,\ell \leq N} \left(\frac{n_k}{n_\ell}\right)^{it} dt \label{anintpro}\\
& = & NT + 2 \sum_{\substack{1 \leq k, \ell \leq N,\\k \neq \ell}} \int_0^T \left(\frac{n_k}{n_\ell}\right)^{it} dt. \label{summand}
\end{eqnarray}
By~\eqref{size} for every $k,\ell$ the integral on the right-hand side of~\eqref{summand} is $\ll T^{\ve}$, since we have $|(\log (n_k/n_\ell))^{-1}| \ll T^{\ve}$. Thus 
\begin{equation} \label{l2n}
\int_0^T |A(t)|^2 \ll NT + \mathcal{O} \left(N^2 T^{\ve}\right) \ll NT,
\end{equation}
which together with~\eqref{l2nan} proves that 
\begin{eqnarray*}
\max_{0 \leq t \leq T} |\zeta(\alpha+it)| & \gg & \exp \left(\frac{c_\alpha (\log N)^{1- \alpha}}{(\log \log N)^\alpha}\right) \\
& \gg & \exp \left(\frac{c_\alpha (\log T)^{1- \alpha}}{\log \log T}\right).
\end{eqnarray*}

If we could increase the value of $M$ (and accordingly, also the value of $N$), then the value of the GCD sum in~\eqref{contri} would also be increased, leading to a larger lower bound for $\max_{0 \leq t \leq T} |\zeta(\alpha+it)|$. With a view to Montgomery's result it would be reasonable to try to increase the value of $M$ from $\mu \log T / \log \log T$ to $\mu \log T$ for some ``small'' $\mu$; this would mean that the Dirichlet polynomial in~\eqref{ant} would still be a sum of $N = 2^M \ll T^\ve$ terms, while, however, some of the terms $b_k$ would be significantly larger than $T$. However, increasing the value of $M$ is highly problematic, since there are several positions where the validity of the argument depends on the specific choice of $M$ which we have made at the beginning. Choosing $M$ significantly larger would
\begin{itemize}
\item lead to a larger error term on the right-hand side of~\eqref{dobs}, which would no longer be dominated by the other terms,
\item cause problems when estimating the contribution of those indices in~\eqref{dobs} for which $m b_k \neq n b_\ell$; it is crucial for the argument that the ratio $m b_k / (n b_\ell)$ is bounded away from 1 whenever $m b_k \neq n b_\ell$, since otherwise we cannot control the contribution of these terms anymore (this problem is reflected in the Montgomery--Vaughan inequality),
\item cause problems in calculating the integral on the left-hand side of~\eqref{anintpro}, since here it is crucial that the ratio $b_k/b_\ell$ is bounded away from 1 for $k \neq \ell$, which requires the estimate~\eqref{size}.
\end{itemize}

Overall, there are several crucial issues to address when attempting to increase the value of $M$ in this argument, all centered on problems around controlling the ratios $b_k/b_\ell$ and $m b_k / n b_\ell$. In the following section we will resolve these issues, by introducing some novel ideas. The most significant ones are a) replacing the finite Euler product $A(t)$ by a different Dirichlet polynomial, whose components are separated away from each other, and which is constructed by selecting some of the elements of the sum~\eqref{ant} as representatives of classes according to a classification with respect to their size, and b) introducing a weight function $w(t)$ to~\eqref{sounint} which guarantees that the contribution to~\eqref{dobs} of those indices $(m,n,k,\ell)$ for which $m b_k /(n b_\ell) \approx 1$ is non-negative.\\

\emph{Remark:~} When I was already finished writing this manuscript I learned about a paper of Voronin~\cite{voronin}, which was published in 1988. In this paper Voronin uses the resonance method in exactly the same way as Hilberdink, so probably proper credit for the development of this method must be given to Voronin. However, it seems that Voronin's paper has never been cited by anybody, and has been totally overlooked by the scientific community. Voronin obtains exactly the same lower bound as Hilberdink (that is, a weaker one than Montgomery), and his argument suffers from the same limitations as Hilberdink's.

\section{Proof of Theorem~\ref{th1}} \label{secproof}

Throughout the rest of this paper, we assume that $\alpha \in (1/2,1)$ is fixed. Let $T$ be given. Constants implied by the symbols ``$\ll$'', ``$\gg$'' and ``$\mathcal{O}$'' may depend on $\alpha$, but not on $T$ or anything else. We will repeatedly assume that $T$ is ``sufficiently large'', which means that we could assume at the very beginning that $T \geq T_0(\alpha)$. We set
\begin{equation} \label{M}
M = \left\lceil (2\alpha-1) \log^{(2)} T \right\rceil.
\end{equation}
Let $(p_r)_{r \geq 1}$ denote the sequence of primes, sorted in increasing order. Let $\mathcal{B}=\{b_1, \dots, b_N\}$ denote the set of the $N=2^M$ positive integers of the form
\begin{equation} \label{prod2}
\prod_{r=1}^M p_{r}^{\beta_r}, \qquad (\beta_1, \dots, \beta_M) \in \{0,1\}^M.
\end{equation}
If a specific number $b_k$ can be represented in the form~\eqref{prod2} for a specific vector $(\beta_1,\dots, \beta_M)$, we say that $(\beta_1,\dots,\beta_M)$ is the exponent vector corresponding to $b_k$. For two numbers $b_k,b_\ell$ with exponent vectors $(\beta_1, \dots, \beta_M)$ and $(\gamma_1, \dots, \gamma_M)$, respectively, we define the distance $\delta(b_k,b_\ell)$ by setting
$$
\delta(b_k,b_\ell) = \sum_{r=1}^M |\beta_r - \gamma_r|.
$$
Thus $\delta(b_k,b_\ell)$ is the number of primes which appear either in the representation of $b_k$ only or in the representation of $b_\ell$ only.\\

As noted in the previous section, one of the main critical issues is that we cannot control the ratios $b_k/b_\ell$ of the elements of $\mathcal{B}$; in principle, it is possible that this ratio is very close to 1 for a large number of pairs $(k,\ell)$. However, the main contribution in the GCD sum does \emph{not} come from those pairs $(b_k,b_\ell)$ whose ratio is close to 1 (bot not equal to 1), because for such a pair the distance $\delta(b_k,b_\ell)$ must be large (it is easily seen that if $b_k$ and $b_\ell$ only have a few different prime factors, then $b_k/b_\ell$ cannot be arbitrarily close to 1). This means that in this case $b_k$ and $b_\ell$ must have many different prime factors, and consequently 
$$
\frac{(\gcd(b_k,b_\ell))^{2 \alpha}}{(b_k b_\ell)^\alpha}
$$
must be small. Therefore, those pairs $(b_k,b_\ell)$ whose ratio $b_k/b_\ell$ is very close to 1 do not contribute significantly to the GCD sum, but they cause severe problems in the upper bound for the square-integral of $A(t)$. As we will show in the sequel we can avoid these problems, roughly speaking by identifying numbers $b_k \neq b_\ell$ whose quotient is very close to 1, and replacing them by a single representative. In the following paragraphs, we will make these observations and definitions precise.\\

The following lemma shows that we can restrict the GCD sum to those pairs $(b_k,b_\ell)$ which have a specific, relatively small distance $\delta$. Recall that by~\eqref{M} there is a direct connection between the values of $T,~M$ and $N$, and that consequently the phrase ``$T$ is sufficiently large'' can also be read as ``$T$, $M$ and $N$ are sufficiently large''. 
\begin{lemma} \label{lemma1}
Set
\begin{equation} \label{R}
R = \left\lfloor \frac{M^{1-\alpha}}{e (\log M + \log \log M)^\alpha} \right\rfloor.
\end{equation}
Then for sufficiently large $T$ for every index $k \in \{1, \dots, N\}$ we have
$$
\sum_{\substack{1 \leq \ell \leq N,\\ \delta(b_k,b_\ell) = R}} \frac{(\gcd(b_k,b_\ell))^{2\alpha}}{(b_k b_\ell)^\alpha} \geq e^{M^{1-\alpha} / (2.72 (\log M)^{\alpha})}.
$$
\end{lemma}
In order to proceed with our proof of Theorem~\ref{th1}, for the moment we will take Lemma~\ref{lemma1} for granted. Its proof will be given, as well as the proofs of all other lemmas stated in the present section, in the subsequent Section~\ref{seclemma}.\footnote{To put the conclusion of Lemma~\ref{lemma1} into relation, we note that for every set $n_1, \dots, n_N$ of distinct positive integers we have $\sum_{1 \leq k,\ell \leq N} \frac{(\gcd(n_k,n_\ell))^{2 \alpha}}{(n_k n_\ell)^\alpha} \ll N \exp \left( \frac{c_\alpha (\log N)^{1-\alpha}}{(\log \log N)^\alpha}\right).$ See~\cite{abs,lewko}. Accordingly our choice of $b_1, \dots, b_N$ and $R$ is essentially optimal, except for the value of the constant in the exponential term.}\\

For $j = 1, 2, \dots$ we define the sets
\begin{equation} \label{bj}
\mathcal{B}_j = \left\{b \in \mathcal{B}:~\left(1+\frac{1}{T}\right)^{j-1} \leq b < \left(1 + \frac{1}{T} \right)^{j} \right\}.
\end{equation}
Then every element of $\mathcal{B}$ is contained in one set $\mathcal{B}_j$, for some appropriate $j$. However, it is possible that some of the sets $\mathcal{B}_j$ contain more than one element of $\mathcal{B}$. If this is the case, then all elements of $\mathcal{B}_j$ have a pairwise ratio which is very close to~1, which would cause serious problems in a calculation along the lines of~\eqref{anintpro}--\eqref{summand}.\footnote{The question whether there really exist sets $\mathcal{B}_j$ containing more than one element, or how many such sets there are, seems to be very difficult; it can be expressed as a problem concerning linear forms in logarithms of integers. We could avoid many of the difficulties in the proof of Theorem~\ref{th1} if we could directly show that $\int_0^W \left|\prod_{r=1}^V (1+p_r^{it})\right|^2 dt \ll 2^V W$ whenever $W \geq e^{c V}$ for some appropriate constant $c$; this would be the necessary upper bound for the square-integral of $A$. Note that the Montgomery--Vaughan mean value 
theorem does not suffice here, since it only works for $W \geq e^{c V \log V}$.} Consequently, we replace all elements of such a set $\mathcal{B}_j$ by one single representative. As the arguments below will show, this allows us to keep the lower bound for the GCD sum, and at the same time also allows us to obtain the desired upper bound for the square-integral of the resonator function. An important observation for the whole argument is the fact that whenever $1 \leq \delta(b_k,b_\ell) \leq 2R$, where $R$ is the number from Lemma~\ref{lemma1}, then $b_k$ and $b_\ell$ cannot be contained in the same set $\mathcal{B}_j$; this observation is stated as a lemma below.
\begin{lemma} \label{lemma1a}
Assume that for two elements $b_k$ and $b_\ell$ of $\mathcal{B}$ we have $1 \leq \delta(b_k,b_\ell) \leq 2R$, where $R$ is defined in~\eqref{R}. Then $b_k$ and $b_\ell$ cannot be contained in the same set $\mathcal{B}_j$ for some $j$, provided that $T$ is sufficiently large.
\end{lemma}

Let $\mathcal{D}$ be the set
\begin{equation*}
\mathcal{D} = \bigcup_{j=1}^\infty~ \min (\mathcal{B}_j).
\end{equation*}
Then $\mathcal{D}$ contains exactly one element from every non-empty set $\mathcal{B}_j$. We set $K = \# \mathcal{D}$, and we write $d_1, \dots, d_K$ for the elements of $\mathcal{D}$, sorted in increasing order. Trivially we have $K \leq N$. As the following Lemma~\ref{lemma1b} shows, whenever there exist indices $k$ and $\ell$ such that $b_k \in \mathcal{B}_j$ and $d_\ell \in \mathcal{B}_j$ for some $j$, then the ratio $b_k/d_\ell$ is relatively close to 1. The subsequent Lemma~\ref{lemma1c} shows that the numbers $d_1,\dots,d_K$ are separated away from each other. Both lemmas follow directly from our definitions, and do not require a detailed proof.

\begin{lemma} \label{lemma1b}
Suppose that there exist $k \in \{1, \dots, N\}$, $l \in \{1, \dots, K\}$, and $j \in \N^+$ such that $b_k \in \mathcal{B}_j$ and $d_\ell \in \mathcal{B}_j$. Then
$$
\frac{b_k}{d_\ell} \in \left[1,1+\frac{1}{T} \right].
$$
\end{lemma}

\begin{lemma} \label{lemma1c}
Let $d_k$ and $d_\ell$ be elements of $\mathcal{D}$, and assume that $k < \ell$. Then we have
$$
\frac{d_\ell}{d_k} \geq \left(1 + \frac{1}{T} \right)^{\ell-k-1}.
$$
\end{lemma}

Now we define a Dirichlet polynomial $A(t)$ by setting
$$
A(t) = \sum_{k=1}^K d_k^{it}.
$$
This definition should be compared with the one in~\eqref{ant} and~\eqref{ant2}. Instead of taking \emph{all} numbers $\mathcal{B} = \{b_1, \dots, b_N\}$ for the definition of $A$ (which means that $A$ is a finite Euler product) we only choose a subset $\mathcal{D} = \{d_1, \dots, d_K\} \subset \mathcal{B}$ of elements which are separated away from each other. Together with Lemma~\ref{lemma1c} this will allow us to obtain the desired upper bound for the square-integral $\int |A(t)|^2~dt$, while preserving the desired lower bound for $\int |\zeta(\alpha+it) A(t)|^2~dt$.\\

By~\eqref{zetaapprox} we have
\begin{equation} \label{zetaapprox2}
\zeta(\alpha+ it) = \sum_{n=1}^T \frac{1}{n^{\alpha+it}} + \mathcal{O} \left(1 \right)
\end{equation}
uniformly for $t \in [T^{1-\alpha},T]$. We can assume that
\begin{equation} \label{canassume}
\max_{0 \leq t \leq T} |\zeta(\alpha+it)| \leq \underbrace{\exp\left(\frac{0.2 (2 \alpha-1)^{1-\alpha} \log T}{\log \log T}\right)}_{=: ~\textup{err}(T)},
\end{equation}
since otherwise the theorem is already proved. Then by~\eqref{zetaapprox2} and~\eqref{canassume} we have
$$
\left|\zeta(\alpha+it)\right|^2 = \left| \sum_{n=1}^{T} \frac{1}{n^{\alpha+it}} \right|^2 + \mathcal{O} \left(\textup{err}(T) \right),
$$
uniformly for $t \in [T^{1-\alpha},T]$. Thus we have
\begin{equation} \label{rhsleft}
|\zeta(\alpha+it) A(t)|^2 = \sum_{1 \leq k,\ell \leq K} \sum_{1 \leq m,n \leq T} \frac{1}{(mn)^\alpha}  \left( \frac{m d_k}{n d_\ell} \right)^{it} + \mathcal{O} \left(\textup{err}(T) |A(t)|^2 \right),
\end{equation}
uniformly for $t \in [T^{1-\alpha},T]$. Since this approximation holds for $t \in [T^{1-\alpha},T]$, we will choose this interval as our range of integration for the resonance method (this is easier to handle than the integration range depending on $m$ and $n$, which appears in~\eqref{dobs}).\\

We can write the double sum on the right-hand side of~\eqref{rhsleft} in the real form
\begin{eqnarray} 
& & \sum_{1 \leq k,\ell \leq K} \sum_{1 \leq m,n \leq T} \frac{1}{(mn)^\alpha} \left( \frac{m d_k}{n d_\ell} \right)^{it} \nonumber\\
& = & \sum_{\substack{1 \leq k, \ell \leq K}} \sum_{\substack{1 \leq m, n \leq T}} \frac{1}{(mn)^\alpha} \cos \left(\left| \log \left(\frac{m d_k}{n d_\ell}\right) \right| t\right).\label{rhswriteex}
\end{eqnarray}
We have to distinguish between three different kinds of frequencies of the cosine-terms appearing in this expression.
\begin{itemize}
\item Frequencies which are very small (type 1 frequencies):~these frequencies occur when $m d_k /(n d_\ell)\approx 1$. These are the ``good'' frequencies which give the positive contributions which we want to have, since in this case $\int_{T^{1-\alpha}}^T \cos \left(\left| \log \left(\frac{m d_k}{n d_\ell}\right) \right| t\right) ~dt \approx T$. The total contribution of these frequencies can be estimates by the connection with GCD sums.\\
\item Frequencies of ``medium'' size (type 2 frequencies): these are terms of the form $\cos a t$ where the frequency $a$ is rather small, but not so close to zero as to guarantee that $\int_{T^{1-\alpha}}^T \cos at ~dt \approx T$. For these frequencies the absolute value of the integral $\int_{T^{1-\alpha}}^T \cos at ~dt$ can be of order roughly $T$, but we cannot be sure whether the integral itself is positive or negative. We solve the problem by replacing the integrand $|\zeta(\alpha+it) A(t)|^2$ by $|\zeta(\alpha+it) A(t)|^2 w(t)$ for an appropriate \emph{weight function} $w$. This will make sure that the contribution of these frequencies is non-negative, without affecting the positive contribution of the ``good'' terms from above.\\
\item Large frequencies (type 3 frequencies): these are terms of the form $\cos a t$ where the frequency $a$ is relatively large. In this case the absolute value of the integral $\int_{T^{1-\alpha}}^T \cos a t ~dt$ is small in comparison with $T$, but there are many combinations of indices $k,\ell,m,n$ leading to such frequencies. We will use an estimate somewhat similar to the Montgomery--Vaughan inequality to guarantee that the total contribution of these terms is sufficiently small.\\
\end{itemize}
For the exact classification of frequencies into types 1--3, we have the following rules.\\
\begin{itemize}
\item Type 1: Frequencies in $\left[0,\frac{1}{T}\right]$.
\item Type 2: Frequencies in $\left(\frac{1}{T},\frac{1}{2T^{1-\alpha}}\right]$
\item Type 3: Frequencies in $\left(\frac{1}{2T^{1-\alpha}},\infty\right)$.
\end{itemize}
This classification will appear again in Lemmas~\ref{lemma2}--\ref{lemma4} below.\\

Based on the remarks made above, instead of~\eqref{sounint} we will rather estimate the integral
\begin{eqnarray}
\int_{T^{1-\alpha}}^{T} |\zeta(\alpha+it) A(t)|^2 w(t) ~dt,  \label{rhswrite2}
\end{eqnarray}
where the weight function $w(t)$ is given by
\begin{equation} \label{w}
w(t) = \left\{ \begin{array}{ll} 3 - \frac{t}{T}  & \textrm{for $t \in [T^{1-\alpha}, 2 T^{1-\alpha}]$}, \\ 1 - \frac{t}{T} & \textrm{for $t \in (2 T^{1-\alpha},T]$} \end{array}\right.
\end{equation}
(we assume that $T$ is sufficiently large such that $2 T^{1-\alpha} \leq T$). By~\eqref{rhsleft} and~\eqref{rhswriteex} we have
\begin{eqnarray}
& & \int_{T^{1-\alpha}}^{T} |\zeta(\alpha+it) A(t)|^2 w(t) ~dt \nonumber\\
& = & \int_{T^{1-\alpha}}^{T} \sum_{\substack{1 \leq k, \ell \leq K}} \sum_{\substack{1 \leq m, n \leq T}} \frac{1}{(mn)^\alpha} \cos \left( \left|\log \left(\frac{m d_k}{n d_\ell}\right) \right| t\right) w(t) ~dt \label{integ}\\
& & \qquad  + \mathcal{O} \left( \textup{err}(T) \int_0^T |A(t)|^2 ~dt \right). \nonumber
\end{eqnarray}
The following three lemmas show how to treat the frequencies of type 1--3 appearing in~\eqref{integ}. By Lemma~\ref{lemma2} the contribution of the ``good'' terms in~\eqref{integ} can be estimated in terms of the value of the GCD sum from Lemma~\ref{lemma1}. The subsequent two lemmas show that the contribution of the frequencies of type 2 is non-negative, and that the total (possibly negative) contribution of the frequencies of type 3 is ``small''.

\begin{lemma}[Lemma for type 1 frequencies] \label{lemma2}
We have
\begin{eqnarray} \label{lemma2equ}
& & \int_{T^{1-\alpha}}^{T} \underbrace{\sum_{\substack{1 \leq k, \ell \leq K}} \sum_{\substack{1 \leq m, n \leq T}}}_{\left| \log \left(\frac{m d_k}{n d_\ell}\right) \right| \leq \frac{1}{T}} \frac{1}{(mn)^\alpha} \cos \left( \left| \log \left(\frac{m d_k}{n d_\ell}\right) \right| t\right) w(t) ~dt \\
& \geq & \frac{KT}{4} \exp\left(\frac{M^{1-\alpha}}{2.72 (\log M)^\alpha} \right), \nonumber
\end{eqnarray}
provided that $T$ is sufficiently large.
\end{lemma}

\begin{lemma}[Lemma for type 2 frequencies] \label{lemma3}
We have
$$
\int_{T^{1-\alpha}}^{T} \underbrace{\sum_{\substack{1 \leq k, \ell \leq K}} \sum_{\substack{1 \leq m, n \leq T}}}_{\left| \log \left(\frac{m d_k}{n d_\ell}\right) \right| \in \left(\frac{1}{T},\frac{1}{2T^{1-\alpha}}\right]} \frac{1}{(mn)^\alpha} \cos \left( \left| \log \left(\frac{m d_k}{n d_\ell}\right) \right| t\right) w(t)~dt \geq 0,
$$
provided that $T$ is sufficiently large.
\end{lemma}

\begin{lemma}[Lemma for type 3 frequencies] \label{lemma4}
We have
$$
\left|  \int_{T^{1-\alpha}}^{T} \underbrace{\sum_{\substack{1 \leq k, \ell \leq K}} \sum_{\substack{1 \leq m, n \leq T}}}_{\left| \log \left(\frac{m d_k}{n d_\ell}\right) \right| \in \left(\frac{1}{2T^{1-\alpha}},\infty\right)} \frac{1}{(mn)^\alpha} \cos \left( \left| \log \left(\frac{m d_k}{n d_\ell}\right) \right| t\right) w(t) ~dt\right| = \mathcal{O} \left(K^2 T^{2-2\alpha} \log T\right).
$$
\end{lemma}

Taking Lemmas~\ref{lemma2}--\ref{lemma4} for granted (the proofs are given in Section~\ref{seclemma} below) we finally obtain
\begin{eqnarray} 
& & \int_{T^{1-\alpha}}^{T} |\zeta(\alpha+it) A(t)|^2 w(t) ~dt \nonumber\\
& \geq & \frac{KT}{4} \exp\left(\frac{M^{1-\alpha}}{2.72 (\log M)^\alpha} \right) + \mathcal{O} \left(K^2 T^{2-2\alpha} \log T\right)+ \mathcal{O} \left(\textup{err}(T) \int_0^T |A(t)|^2 ~dt \right),  \label{lowerbound}
\end{eqnarray}
provided that $T$ is sufficiently large.\\

Now we will prove an upper bound for the square-integral of $A(t)$. By Lemma~\ref{lemma1c} we have
$$
|\log (d_k/d_\ell)| \geq \left(|k - \ell| - 1\right) \log \left(1+\frac{1}{T}\right) \geq \frac{|k - \ell| - 1}{2 T},
$$
provided that $T$ is sufficiently large. Thus we have
\begin{eqnarray}
\int_0^T |A(t)|^2~dt & = & \sum_{1 \leq k,\ell \leq K} \int_0^T \left(\frac{d_k}{d_\ell}\right)^{it}~dt \nonumber\\
& \leq & 3KT + \sum_{\substack{1 \leq k, \ell \leq K,\\|k-\ell| \geq 2}} \frac{2}{|\log (d_k/d_\ell)|} \nonumber\\
& \leq & 3KT + \sum_{1 \leq k \leq K} \sum_{v=1}^K \frac{8T}{v}, \nonumber\\
& \leq & 11KT (1+\log K), \label{l2upper}
\end{eqnarray}
provided that $T$ is sufficiently large. Note that by~\eqref{M} we have 
\begin{equation} \label{mup}
M^{1-\alpha} / (2.72 (\log M)^{\alpha}) > 0.36 ((2 \alpha-1)\log T)^{1-\alpha}/(\log \log T)^{\alpha}
\end{equation}
for sufficiently large $M$. Consequently, by~\eqref{canassume} and~\eqref{l2upper}, the third term in line~\eqref{lowerbound} is dominated by the first, and we have
\begin{equation} \label{forth2}
\int_{T^{1-\alpha}}^{T} |\zeta(\alpha+it) A(t)|^2 w(t) ~dt \geq \frac{KT}{5} e^{M^{1-\alpha} / (2.72 (\log M)^\alpha)}
\end{equation}
for sufficiently large $T$. Together with~\eqref{l2upper} and~\eqref{mup} this implies the existence of a value of $t \in [T^{1-\alpha},T]$ for which
$$
|\zeta(\alpha+it)|^2 \geq \exp \left(\frac{0.36 ((2 \alpha-1)\log T)^{1-\alpha}}{(\log \log T)^{\alpha}} \right),
$$
provided that $T$ is sufficiently large. This proves Theorem~\ref{th1}.

\section{Proofs of auxiliary lemmas for Theorem~\ref{th1}} \label{seclemma}

For the proofs in this section we will use the following form of the prime number theorem.
\begin{lemma}[{\cite[Theorem 8.8.4]{bach}}] \label{lemmabach}
Let $p_r$ denote the $r$-th prime number. Then we have
$$
p_r < r (\log r + \log \log r), \qquad \textrm{for $r \geq 6$.}
$$
\end{lemma}

\begin{proof}[Proof of Lemma~\ref{lemma1}]
By Lemma~\ref{lemmabach} we have
\begin{equation} \label{primesize}
 p_{r} \leq M (\log M + \log \log M) \qquad \textrm{for $1 \leq r \leq M$},
\end{equation}
provided that $M$ is sufficiently large. For given $k \in \{1, \dots, N\}$ there exists $\binom{M}{R}$ indices $\ell \in \{1, \dots, N\}$ such that $\delta(b_k,b_\ell)=R$. Thus by~\eqref{primesize} we have
\begin{eqnarray}
& & \sum_{\substack{1 \leq \ell \leq N,\\ \delta(b_k,b_\ell)=R}} \frac{(\gcd(b_k,b_\ell))^{2 \alpha}}{(b_k b_\ell)^\alpha} \nonumber\\
& \geq & \sum_{\substack{1 \leq \ell \leq N,\\ \delta(b_k,b_\ell)=R}} (M (\log M+\log \log M))^{-\alpha R} \nonumber\\
& = & \binom{M}{R} (M (\log M+\log \log M))^{-\alpha R} \label{gcdlob1} \\
& \geq & \frac{M^{M+1/2}}{3 (M-R)^{M-R+1/2} R^{R+1/2}  (M (\log M+\log \log M))^{\alpha R}} \label{gcdlob2} \\
& \geq & \frac{1}{3\sqrt{M}} ~\exp \left(R \log M - R \log R - \alpha R \log (M (\log M + \log \log M)) \right), \label{gcdlob3}
\end{eqnarray}
where to get from~\eqref{gcdlob1} to~\eqref{gcdlob2} we used a version of Stirling's formula which states that
$$
n^{n+1/2} \exp\left(-n+\frac{1}{12n+1} \right) < \frac{n!}{\sqrt{2 \pi}} < n^{n+1/2} \exp \left(-n+ \frac{1}{12 n}\right), \qquad \textrm{for $n = 1,2,\dots$}
$$
(see~\cite{robbins}). For the specific choice of $R$ from~\eqref{R} we have
\begin{eqnarray*}
& & R \log M - R \log R - \alpha R \log (M (\log M + \log \log M)) \\
& \geq & R \log M - (1-\alpha) R \log M + R + \alpha R \log (\log M + \log \log M) \nonumber\\
& & \qquad - \alpha R \log M - \alpha R \log( \log M + \log \log M) \\
& = & R.
\end{eqnarray*}
Thus by~\eqref{gcdlob3} we have
\begin{eqnarray*}
\sum_{\substack{1 \leq \ell \leq N,\\ \delta(b_k,b_\ell)=R}} \frac{(\gcd(b_k,b_\ell))^{2 \alpha}}{(b_k b_\ell)^\alpha} & \geq & \frac{1}{3 \sqrt{M}} e^R \\
& \geq & e^{M^{1-\alpha} / (2.72 (\log M)^{\alpha})}
\end{eqnarray*}
for sufficiently large $T$, which proves the lemma.
\end{proof}

\begin{proof}[Proof of Lemma~\ref{lemma1a}]
W.l.o.g.\ we assume that $b_k \geq b_\ell$. By assumption we have $1 \leq \delta(b_k,b_\ell) \leq 2R$, which implies that $b_k \neq b_\ell$, and which also implies that the denominator of the reduced fraction $b_k/b_\ell$ is the product of a most $2R$ different primes. By Lemma~\ref{lemmabach} each of these primes is bounded above by $M (\log M + \log \log M)$, provided that $T$ is sufficiently large. Consequently, the denominator of the reduced fraction $b_k/b_\ell$ is bounded above by
\begin{equation} \label{lemma1ae}
\left(M (\log M + \log \log M)\right)^{2R} \leq e^{3 R \log M} \leq e^{3 M^{1-\alpha} (\log M)^{1-\alpha}} \leq \sqrt{T}
\end{equation}
for sufficiently large $T$, where we used~\eqref{M}. Thus the ratio $b_k/b_\ell$ cannot be smaller than $(\sqrt{T}+1)/\sqrt{T} = 1+1/\sqrt{T}$. However, by construction all numbers which are contained in the same set $\mathcal{B}_j$ for some $j$ have a ratio which is smaller than $1+1/T$. Consequently, $b_k$ and $b_\ell$ cannot be contained in the same set $\mathcal{B}_j$, provided that $T$ is sufficiently large.
\end{proof}

\emph{Lemma~\ref{lemma1b} and Lemma~\ref{lemma1c}:} As noted before their statement, these lemmas follow directly from our definitions and do not require a detailed proof.

\begin{proof}[Proof of Lemma~\ref{lemma2}]
Note that whenever 
$$
\left| \log \left(\frac{m d_k}{n d_\ell}\right) \right| \leq \frac{1}{T},
$$
then
$$
\cos \left( \left| \log \left(\frac{m d_k}{n d_\ell}\right) \right| t\right) \geq 0 \qquad \textrm{for} \qquad t \in [0,T],
$$
and consequently 
\begin{equation} \label{nonneg}
\int_{T^{1-\alpha}}^{T} \cos \left( \left| \log \left(\frac{m d_k}{n d_\ell}\right) \right| t\right) w(t) ~dt \geq 0.
\end{equation}
Thus all summands in~\eqref{lemma2equ} are non-negative, and we can simply neglect those which we do not want to account for.\\

Throughout this proof, let $k \in \left\{1, \dots, K\right\}$ be fixed. For this $k$, we want to calculate 
\begin{equation} \label{lemma2e}
\sum_{\ell=1}^K  \sum_{\substack{1 \leq m, n \leq T,\\ \left| \log \left(\frac{m d_k}{n d_\ell}\right) \right| \leq \frac{1}{T}}} \int_{T^{1-\alpha}}^T \frac{1}{(mn)^\alpha} \cos \left(\left| \log \left(\frac{m d_k}{n d_\ell}\right) \right| t\right) w(t) ~dt.
\end{equation}
There exists a $j \in \N^+$ such that $d_k \in \mathcal{B}_j$. By construction there also exists an element of $\mathcal{B}$ which is equal to $d_k$; in other words, there exists an index $\bar{k} \in \{1, \dots N\}$ such that $b_{\bar{k}} = d_k$. There exist $\binom{M}{R}$ elements of $\mathcal{B}$ whose distance $\delta$ from $b_{\bar{k}}$ is precisely $R$, where $R$ is the number from~\eqref{R}. We denote them by $b_{\bar{k}_1}, \dots, b_{\bar{k}_H}$, where $H = \binom{M}{R}$. There also exist indices $j_1, \dots, j_H \in \N^+$ such that
$$
b_{\bar{k}_1} \in \mathcal{B}_{j_1}, \dots, b_{\bar{k}_H} \in \mathcal{B}_{j_H}.
$$
If we pick two numbers from $b_{\bar{k}_1}, \dots, b_{\bar{k}_H}$, then their distance $\delta$ can be at most $2R$; consequently, by Lemma~\ref{lemma1a}, all indices $j_1, \dots, j_H$ must be different. Furthermore, also by Lemma~\ref{lemma1a}, all the indices $j_1, \dots, j_H$ must be different from $j$. Consequently, there exist indices $k_1, \dots, k_H \in \{1, \dots, K\}$, all of which are different from each other and are different from $k$, such that 
$$
d_{k_1} \in \mathcal{B}_{j_1}, \dots, d_{k_H} \in \mathcal{B}_{j_H}.
$$

By assumption, $\delta(b_{\bar{k}}, b_{\bar{k}_1}) = R$. This implies that both the numbers
$$
\frac{b_{\bar{k}}}{\gcd(b_{\bar{k}},b_{\bar{k_1}})} \qquad \textrm{and} \qquad \frac{b_{\bar{k}_1}}{\gcd(b_{\bar{k}},b_{\bar{k_1}})}
$$
are the product of at most $R$ different prime factors. By Lemma~\ref{lemmabach} each of those prime factors is bounded above by $M (\log M + \log \log M)$, assuming that $T$ is sufficiently large. As in the calculation following~\eqref{lemma1ae}, this implies that
\begin{equation} \label{mn}
m_1:= \frac{b_{\bar{k}_1}}{\gcd(b_{\bar{k}},b_{\bar{k_1}})} \leq T, \qquad n_1 := \frac{b_{\bar{k}}}{\gcd(b_{\bar{k}},b_{\bar{k_1}})} \leq T,
\end{equation}
always assuming that $T$ is sufficiently large. Furthermore, we have 
\begin{equation} \label{mn2}
m_1 b_{\bar{k}} = n_1 b_{\bar{k}_1}.
\end{equation}
By Lemma~\ref{lemma1a} we have 
$$
\frac{b_{\bar{k}_1}}{d_{k_1}} \in \left[1, 1+\frac{1}{T} \right].
$$
Thus by~\eqref{mn2}, and recalling that $d_k = b_{\bar{k}}$, we have
$$
\frac{m_1 d_{k}}{n_1 d_{k_1}} \in \left[1, \left(1+\frac{1}{T}\right) \right],
$$
which implies that
\begin{equation} \label{lowerfreq}
\left| \log \left(\frac{m_1 d_k}{n_1 d_{k_1}}\right) \right| \leq \frac{1}{T}.
\end{equation}
Consequently for the contribution of the pair of indices $(k,k_1)$ to the sum in~\eqref{lemma2e}, by~\eqref{nonneg},~\eqref{mn} and~\eqref{lowerfreq} we have
\begin{eqnarray*}
& & \sum_{\substack{1 \leq m, n \leq T,\\ \left| \log \left(\frac{m d_k}{n d_{k_1}}\right) \right| \leq \frac{1}{T}}} \int_{T^{1-\alpha}}^T \frac{1}{(mn)^\alpha} \cos \left(\log \left| \left(\frac{m d_{k}}{n d_{k_1}}\right) \right| t\right)   w(t)~dt \\
& \geq & \frac{1}{(m_1 n_1)^\alpha} \int_{T^{1-\alpha}}^T \cos \left(\log \left| \left(\frac{m_1 d_k}{n_1 d_{k_1}}\right) \right| t\right) w(t)~dt \\
& \geq & \frac{(\gcd(b_{\bar{k}},b_{\bar{k_1}}))^{2 \alpha}}{(b_{\bar{k}} b_{\bar{k}_1})^\alpha} \underbrace{\int_{T^{1-\alpha}}^T (\cos 1) \left(1 - \frac{t}{T} \right) ~dt}_{\geq T/4 \textrm{ for sufficiently large $T$}}.
\end{eqnarray*}
In the same way we get a lower bound for the contribution of the pairs of indices $(k,k_2), \dots, (k,k_H)$. Note again, as mentioned above, that all indices $k_1, \dots, k_H$ are different from each other. Consequently, for fixed $k \in \{1, \dots, K\}$, the expression in~\eqref{lemma2e} is bounded below by
\begin{eqnarray*}
& & \sum_{\ell \in \{k_1, \dots, k_H\}} \sum_{\substack{1 \leq m, n \leq T,\\ \left| \log \left(\frac{m d_k}{n d_\ell}\right) \right| \leq \frac{1}{T}}} \int_{T^{1-\alpha}}^T \frac{1}{(mn)^\alpha} \cos \left(\left| \log \left(\frac{m d_k}{n d_\ell}\right) \right| t\right) w(t) ~dt \\
& \geq & \frac{T}{4} \sum_{\ell \in \{\bar{k}_1, \dots, \bar{k}_H \}} \frac{(\gcd(b_{\bar{k}},b_{\ell}))^{2 \alpha}}{(b_{\bar{k}} b_{\ell})^\alpha} \\
& = & \frac{T}{4} \sum_{\substack{1 \leq \ell \leq N,\\ \delta(b_{\bar{k}},b_{\ell})=R}} \frac{(\gcd(b_{\bar{k}},b_{\ell}))^{2 \alpha}}{(b_{\bar{k}} b_{\ell})^\alpha}.
\end{eqnarray*}
Consequently by Lemma~\ref{lemma1} we obtain the result that, still for fixed $k$, the expression in~\eqref{lemma2e} is bounded below by
$$
\frac{T}{4} \exp\left(\frac{M^{1-\alpha}}{2.72 (\log M)^\alpha} \right).
$$
Finally, summing over all possible values of $k$, we arrive at the conclusion of Lemma~\ref{lemma2}.
\end{proof}

\begin{proof}[Proof of Lemma~\ref{lemma3}]
The construction of the weight function $w(t)$ is based on the fact that we have
\begin{equation} \label{cosa}
\int_0^T (\cos a t) \left(1 - \frac{t}{T} \right) ~dt = \frac{1 - \cos a T}{a^2 T} \geq 0
\end{equation}
for all $a > 0$. Thus, under the assumption that $0 < a \leq 1 / (2 T^{1-\alpha})$, we have
\begin{eqnarray*}
& & \int_{T^{1-\alpha}}^{T} (\cos a t) w(t) ~dt \label{w0}\\
& = & \int_{T^{1-\alpha}}^{T} (\cos a t) \left( 1 - \frac{t}{T} \right) dt + 2 \int_{T^{1-\alpha}}^{2 T^{1-\alpha}}  \underbrace{\cos a t}_{\geq~ \cos 1 ~\geq ~0.54} ~dt \\
& \geq & - \int_0^{T^{1-\alpha}} (\cos a t) \left(1 - \frac{t}{T} \right) ~dt  + T^{1-\alpha} \\
& \geq & 0.
\end{eqnarray*}
As a consequence we obtain
\begin{eqnarray*}
& & \int_{T^{1-\alpha}}^{T} \underbrace{\sum_{\substack{1 \leq k, \ell \leq K}} \sum_{\substack{1 \leq m, n \leq T}}}_{\left| \log \left(\frac{m d_k}{n d_\ell}\right) \right| \in \left(\frac{1}{T},\frac{1}{2 T^{1-\alpha}}\right]} \frac{1}{(mn)^\alpha} \cos \left( \left| \log \left(\frac{m d_k}{n d_\ell}\right) \right| t\right) w(t) ~dt  \geq 0, 
\end{eqnarray*}
which proves the lemma.
\end{proof}

\begin{proof}[Proof of Lemma~\ref{lemma4}]
Let $k,\ell \in \{1, \dots, K\}$ be fixed. The proof boils down to finding an upper bound for 
\begin{equation} \label{returi}
\sum_{\substack{1 \leq m,n \leq T, \\\left| \log \left(\frac{m d_k}{n d_\ell}\right) \right| \in \left(\frac{1}{2T^{1-\alpha}},\infty\right)}} \frac{1}{m^\alpha n^\alpha} \frac{1}{\left| \log \left(\frac{m d_k}{n d_\ell}\right) \right|}.
\end{equation}
Proving this upper bound is a somewhat tedious exercise, but does not require any novel ideas; it can be seen as a variant of the Montgomery--Vaughan inequality, where all indices $(m,n)$ accounting for large contributions are left out. We start by assuming w.l.o.g.\ that $k \geq \ell$; that is, we also have $d_k \geq d_\ell$. In a first step we investigate the contribution of those indices $m,n$ for which $n \geq m d_k/d_\ell$. We have
\begin{eqnarray}
& & \sum_{\substack{1 \leq m,n \leq T, \\\left| \log \left(\frac{m d_k}{n d_\ell}\right) \right| \in \left(\frac{1}{2T^{1-\alpha}},\infty\right),\\n \geq m d_k / d_\ell}} \frac{1}{m^\alpha n^\alpha} \frac{1}{\left| \log \left(\frac{m d_k}{n d_\ell}\right) \right|} \label{desest}\\
& = & \sum_{\substack{1 \leq m,n \leq T, \\ \log \left(\frac{n d_\ell}{m d_k}\right) \in \left(\frac{1}{2T^{1-\alpha}},\infty\right),\\ m d_k / d_\ell \leq n \leq 2 m d_k / d_\ell}} \frac{1}{m^\alpha n^\alpha} \frac{1}{\log \left(\frac{n d_\ell}{m d_k}\right)} + \sum_{\substack{1 \leq m,n \leq T, \\ \log \left(\frac{n d_\ell}{m d_k}\right) \in \left(\frac{1}{2T^{1-\alpha}},\infty\right),\\ 2 m d_k / d_\ell < n}} \frac{1}{m^\alpha n^\alpha} \frac{1}{\log \left(\frac{n d_\ell}{m d_k}\right)} \label{secs}
\end{eqnarray}
For the second sum in~\eqref{secs} we have
\begin{equation} \label{comb1}
\sum_{\substack{1 \leq m,n \leq T, \\ \log \left(\frac{n d_\ell}{m d_k}\right) \in \left(\frac{1}{2T^{1-\alpha}},\infty\right),\\ 2 m d_k / d_\ell < n}} \frac{1}{m^\alpha n^\alpha} \underbrace{\frac{1}{\log \left(\frac{n d_\ell}{m d_k}\right)}}_{\leq 1 / \log 2}  \ll \sum_{1 \leq m,n, \leq T}  \frac{1}{m^\alpha n^\alpha} \ll T^{2 - 2 \alpha}. 
\end{equation}
The first sum in~\eqref{secs} is more difficult to estimate. Note that $\log \left(\frac{n d_\ell}{m d_k}\right) > \frac{1}{2T^{1-\alpha}}$ implies that $n d_\ell / (m d_k) > 1 + 1/(2 T^{1-\alpha})$. Thus, for sufficiently large $T$ we have
\begin{eqnarray}
& & \sum_{\substack{1 \leq m,n \leq T, \\ \log \left(\frac{n d_\ell}{m d_k}\right) \in \left(\frac{1}{2T^{1-\alpha}},\infty\right),\\ m d_k / d_\ell \leq n \leq 2 m d_k / d_\ell}} \frac{1}{m^\alpha n^\alpha} \frac{1}{\log \left(\frac{n d_\ell}{m d_k}\right)} \nonumber\\
& \leq & \sum_{1 \leq m \leq \frac{T d_\ell}{d_k}} \frac{1}{m^\alpha} \sum_{1  \leq r \leq 2 + (1-\alpha) \log^{(2)} T} \sum_{\substack{1 \leq n \leq T,\\n \in \left( \frac{m d_k}{d_\ell} \left( 1 + \frac{2^{r-1}}{2 T^{1-\alpha}}\right),\frac{m d_k}{d_\ell} \left( 1 + \frac{2^{r}}{2 T^{1-\alpha}}\right)\right]}} \frac{1}{n^\alpha} \frac{1}{\log \left( 1 + \frac{2^{r-1}}{2 T^{1-\alpha}}\right)} \nonumber\\
& \ll & \sum_{1 \leq m \leq \frac{T d_\ell}{d_k}} \frac{1}{m^\alpha} \sum_{1  \leq r \leq 2 + (1-\alpha) \log^{(2)} T} \left( \frac{m d_k}{d_\ell} \frac{2^{r-1}}{2 T^{1-\alpha}} + 1 \right) \left(\frac{m d_k}{d_\ell}\right)^{-\alpha} \frac{T^{1-\alpha}}{2^{r-1}} \nonumber\\
& = & \sum_{1 \leq m \leq \frac{T d_\ell}{d_k}} \frac{1}{m^\alpha} \sum_{1  \leq r \leq 2 + (1-\alpha) \log^{(2)} T} \left(\frac{m d_k}{d_\ell}\right)^{1-\alpha} \frac{2^{r-1}}{2 T^{1-\alpha}} \frac{T^{1-\alpha}}{2^{r-1}} \nonumber\\
& & \qquad +  \sum_{1 \leq m \leq \frac{T d_\ell}{d_k}} \frac{1}{m^\alpha} \sum_{1  \leq r \leq 2 + (1-\alpha) \log^{(2)} T} \left(\frac{m d_k}{d_\ell}\right)^{-\alpha} \frac{T^{1-\alpha}}{2^{r-1}} \nonumber\\
& \ll & \left(\frac{d_k}{d_\ell}\right)^{1-\alpha} (\log T) \sum_{1 \leq m \leq \frac{T d_\ell}{d_k}} m^{1-2 \alpha} + T^{1-\alpha} \sum_{1 \leq m \leq \frac{T d_\ell}{d_k}} m^{-2\alpha} \nonumber\\
& \ll & T^{2-2\alpha} \log T. \label{comb2}
\end{eqnarray}
Combining~\eqref{comb1} and~\eqref{comb2} we see that the sum in~\eqref{desest} is of order $\ll T^{2-2\alpha} \log T$. Now it is necessary to estimate the contribution to the sum in~\eqref{returi} of those pairs of indices $(m,n)$ for which $n \leq m d_k/d_\ell$. This can be done in the same way as in our calculations for the case $n \geq m d_k/d_\ell$ (just in a sort of mirror-inverted way), and we omit the details. Overall we have
\begin{equation} \label{logT}
\sum_{\substack{1 \leq m,n \leq T, \\\left| \log \left(\frac{m d_k}{n d_\ell}\right) \right| \in \left(\frac{1}{2T^{1-\alpha}},\infty\right)}} \frac{1}{m^\alpha n^\alpha} \frac{1}{\left| \log \left(\frac{m d_k}{n d_\ell}\right) \right|} \ll T^{2-2\alpha} \log T.
\end{equation} 

Recalling the definition of the weight function $w$ (see~\eqref{w}), for any $a \geq 0$ we have
\begin{eqnarray*}
& & \left|  \int_{T^{1-\alpha}}^{T}  \left( \cos a t\right) w(t) ~dt\right| \\
& = & \left| \int_{T^{1-\alpha}}^{T} \left( \cos a t\right) \left(1 - \frac{t}{T} \right) ~dt + 2 \int_{T^{1-\alpha}}^{2 T^{1-\alpha}} \left(  \cos a t\right) ~dt\right| \\
& = & \frac{\cos (a T^{1-\alpha}) - \cos (a T) + a (T^{1-\alpha} - T) \sin (a T^{1-\alpha})}{a^2 T} + 2 \frac{\sin (aT) - \sin (a T^{1-\alpha})}{a} \\
& \leq & \frac{2}{a^2 T} + \frac{5}{a},
\end{eqnarray*}
where the last expression is bounded by $7/a$ if we assume that $a \geq 1/T$. Thus by~\eqref{logT} we have
\begin{eqnarray*}
& & \left| \int_{T^{1-\alpha}}^{T} \underbrace{\sum_{\substack{1 \leq k, \ell \leq K}} \sum_{\substack{1 \leq m, n \leq T}}}_{\left| \log \left(\frac{m d_k}{n d_\ell}\right) \right| \in \left(\frac{1}{2T^{1-\alpha}},\infty\right)} \frac{1}{(mn)^\alpha} \cos \left( \left| \log \left(\frac{m d_k}{n d_\ell}\right) \right| t\right) w(t) ~dt\right| \\
& \leq & 7 \sum_{\substack{1 \leq k, \ell \leq K}} \sum_{\substack{1 \leq m, n \leq T, \\ \left| \log \left(\frac{m d_k}{n d_\ell}\right) \right| \in \left(\frac{1}{2T^{1-\alpha}},\infty\right)}} \frac{1}{m^\alpha n^\alpha} \frac{1}{\left| \log \left(\frac{m d_k}{n d_\ell}\right) \right|} \\
& \ll & K^2 T^{2-2\alpha} \log T,
\end{eqnarray*}
which proves the lemma.
\end{proof}

\section{Proof of Theorem~\ref{th2}} \label{secth2}
We use a standard method for estimating the measure of those $t$ for which $|\zeta(\alpha+it)|$ is ``large''. Let $\tau$ satisfying~\eqref{tau} be given. We use exactly the same construction as in the proof of Theorem~\ref{th1}, but with the definition of $M$ in~\eqref{M} replaced by 
$$
M = \left\lceil \beta \log^{(2)} T \right\rceil,
$$
where we chose
$$
\beta = \beta(\tau) = \left(6 \tau\right)^{\frac{1}{1-\alpha}}.
$$
Here, and in the sequel, $\log^{(2)}$ denotes the logarithm in base 2. Then by~\eqref{tau} we have $\beta \leq 2 \alpha -1$, which means that everything works out as in the proof of Theorem~\ref{th1}. As in the proof of Theorem~\ref{th1} we obtain
\begin{equation} \label{ok1}
\int_0^T |\zeta(\alpha+it) A(t)|^2 \geq \frac{K T}{5} \exp\left(\frac{M^{1-\alpha}}{2.72 (\log M)^\alpha} \right)
\end{equation}
and
\begin{equation} \label{ok}
\int_0^T |A(t)|^2~dt = \mathcal{O} \left( KT (1+\log K) \right)
\end{equation}
for sufficiently large $T$. Let $F_{\alpha,\tau}$ be defined as in~\eqref{falpha}. Then we have
\begin{eqnarray}
&& \int_0^T |(\zeta(\alpha+it) A(t)|^2 ~dt \nonumber\\
& = & \int_{F_{\alpha,\tau}} |(\zeta(\alpha+it) A(t)|^2 ~dt + \int_{[0,T] \backslash F_{\alpha,\tau}} |(\zeta(\alpha+it) A(t)|^2 ~dt. \label{integralsb}
\end{eqnarray}
The second integral in~\eqref{integralsb} is bounded above by
$$
\exp \left(\frac{2 \tau (\log T)^{1 - \alpha}}{(\log \log T)^\alpha} \right) \underbrace{\int_0^T |A(t)|^2~dt}_{= \mathcal{O} \left(KT (1+\log K)\right) \textrm{ by~\eqref{ok}}}.
$$
By our choice of $\beta$ we have $\beta^{1-\alpha}/2.72 > 2 \tau$, which, in view of~\eqref{ok1}, implies that the main contribution to the sum of the two integrals in~\eqref{integralsb} cannot come from the second integral (provided that $T$ is sufficiently large). As a consequence we must have
$$
\int_{F_{\alpha,\tau}} |(\zeta(\alpha+it) A(t)|^2 ~dt \gg KT \exp\left(\frac{M^{1-\alpha}}{2.72 (\log M)^\alpha} \right).
$$
Since for $t \in [0,T]$ we have $|\zeta(\alpha+it)|^2 \ll T^{2-2\alpha}$ and $|A(t)|^2 \leq K^2 \leq K 2^M \ll K T^\beta$, this means that
$$
\textup{meas} (F_{\alpha,\tau}) \geq \frac{K T}{T^{2 - 2 \alpha} K T^{\beta}} \geq T^{2 \alpha - 1 - \beta}
$$
for sufficiently large $T$. This proves Theorem~\ref{th2}. Note that under the Riemann Hypothesis, using~\eqref{riemupp}, we would get a much better bound for the measure of $F_{\alpha,\tau}$. 

\section{Concluding remarks} \label{secconc}

It is interesting that Soundararajan was able to recapture Montgomery's lower bounds without the necessity of applying such an involved construction as the one in the present paper. This is due to the fact that in the case $\alpha = 1/2$ it is sufficient to use a Dirichlet polynomial of length $\ll T^\mu$ for some (small) $\mu$ as a resonator function. In contrast, in our construction for the proof of Theorem~\ref{th1} we could use a Dirichlet polynomial which is the sum of $\ll T^\mu$ terms, but the length of the polynomial is much larger than $T$ (namely, of order roughly $T^{\mu \log T}$). This difference between the cases $\alpha=1/2$ and $\alpha \in (1/2,1)$ also appears in the problem concerning GCD sums; more precisely, in the two problems concerning the maximal order of the sums
$$
\sum_{1 \leq k,\ell \leq N} a_k a_\ell \frac{(\gcd(k,\ell))^{2 \alpha}}{(k \ell)^\alpha} \qquad \textrm{and} \qquad \sum_{1 \leq k, \ell \leq N} a_k a_\ell \frac{(\gcd(n_k,n_\ell))^{2 \alpha}}{(n_k n_\ell)^\alpha},
$$
respectively. Here in both instances we require that $\sum_{k=1}^N a_k^2 \leq 1$, and in the second instance the numbers $\{n_1, \dots, n_N\}$ may be any set of distinct positive integers. In the case $\alpha \in (1/2,1)$ the optimal upper bounds for these problems are significantly different; they are roughly
$$
\exp \left( \frac{c_\alpha (\log N)^{1-\alpha}}{\log \log N}\right) \qquad \textrm{and} \qquad \exp \left( \frac{c_\alpha (\log N)^{1-\alpha}}{(\log \log N)^{\alpha}}\right),
$$
respectively (see~\cite{hilber} for the first result, and~\cite{abs,lewko} for the second). This reflects precisely the difference between the lower bounds for the maximum of the Riemann zeta function obtained by Hilberdink and the ones obtained in the present paper. In the case $\alpha=1/2$, somewhat unexpectedly, the gap between the upper bounds for these two problems disappears: by~\cite{hilber}, the optimal upper bound for the first problem in the case $\alpha=1/2$ is
$$
\exp \left(\sqrt{\frac{c \log N}{\log \log N}} \right), 
$$
and the optimal upper bound for the second problem is most likely of the same order (this problem is not completely solved yet, and an additional multiplicative factor of order at most $\sqrt{\log \log \log N}$ may be necessary in the exponential term; see~\cite{bond}). This is an interesting phenomenon, which deserves further investigation.\\

We add a few remarks concerning possible improvements of our method. Of course it would be desirable to further increase the value of $M$, which would result in an improvement of Theorem~\ref{th1}. In this context, it is interesting to note that our choice of the size of $M$ is not nearly critical in the argument surrounding the replacement of the set $\mathcal{B}$ by $\mathcal{D}$. The critical condition in this part of the argument is (roughly speaking) that $M^R \ll T$. The ``correct'' choice of $R$ is $R \approx M^{1-\alpha}/(\log M)^\alpha$, and thus we require that $e^{M^{1-\alpha} (\log M)^{1 - \alpha}} \ll T$. This would allow a much larger choice of $M$ than the choice $M \approx \log T$ which we have actually made in~\eqref{M}. Instead, what is really limiting the size of~\eqref{M} in our argument is the contribution of the ``type 2'' and ``type 3'' frequencies in Lemmas~\ref{lemma3} and~\ref{lemma4}. This contribution is difficult to estimate, in particular in the case of the large number of 
summands in Lemma~\ref{lemma4}. On the other hand, one would expect that this contribution is somewhat irregular and that positive and negative contributions cancel out to a certain degree. Accordingly, a more detailed analysis in Lemma~\ref{lemma3} and~\ref{lemma4}, or possibly a randomized definition of the integration range and/or of the sets $\mathcal{B}$ and $\mathcal{D}$, could lead to a better result.

\section*{Note added in proof}

In a recent manuscript, Bondarenko and Seip \cite{bond2} proved that \eqref{montlow} can be improved to
\begin{equation} \label{imprlo}
\max_{0 \leq t \leq T} \left| \zeta(1/2+it) \right| = \Omega \left( \exp \left(\frac{c \sqrt{\log T \log \log \log T}}{\sqrt{\log \log T}} \right) \right) \qquad \textrm{as $t \to \infty$.}
\end{equation}
Their argument uses ideas from the present paper, but their proof is much less cumbersome. In particular, they return from the second moment of $\zeta$ (as in the present paper and in Hilberdink's paper) to the first moment (as in Soundararajan's paper), which allows them to use a smoothing function (as in Soundararajan's paper) and to avoid the intricated classification of frequencies as in Section~\ref{secproof} of the present paper. Still, the improved lower bound~\eqref{imprlo} in~\cite{bond2} is not so much a consequence of the improved resonance argument, but follows from improved lower bounds for GCD sums for parameter $\alpha=1/2$.\\

I want to thank the referee for reading my paper very carefully. The referee notes that the method depends heavily on the positivity of the coefficients of the Dirichlet poylnomial which is used to approximate the zeta function, and asks the following questions: Is it possible to establish results for more general L-functions? Is it possible to obtain results in short intervals? Can this method be adapted to extend the results of Lamzouri \cite{lamz} on the value distribution of $\zeta(\alpha+it)$ at the extreme tails of the distributions?

\section*{Acknowledgments}

The research work for this paper was carried out during the year which I spent as a guest at the University of Kobe in Japan. I want to thank Katusi Fukuyama and the administrative staff of the Department of Mathematics of Kobe University, who made this stay not only possible, but turned it into a very pleasant one for me and my family.


\begin{thebibliography}{10}

\bibitem{abs}
C.~Aistleitner, I.~Berkes, and K.~Seip.
\newblock G{CD} sums from {P}oisson integrals and systems of dilated functions.
\newblock {\em J. Eur. Math. Soc. (JEMS)}, 17(6):1517--1546, 2015.

\bibitem{abs2}
C.~Aistleitner, I.~Berkes, K.~Seip, and M.~Weber.
\newblock Convergence of series of dilated functions and spectral norms of
  {GCD} matrices.
\newblock {\em Acta Arith.}, 168(3):221--246, 2015.

\bibitem{bach}
E.~Bach and J.~Shallit.
\newblock {\em Algorithmic number theory. {V}ol. 1}.
\newblock Foundations of Computing Series. MIT Press, Cambridge, MA, 1996.

\bibitem{bala}
R.~Balasubramanian and K.~Ramachandra.
\newblock On the frequency of {T}itchmarsh's phenomenon for {$\zeta (s)$}.
  {III}.
\newblock {\em Proc. Indian Acad. Sci. Sect. A}, 86(4):341--351, 1977.

\bibitem{bond2}
A.~Bondarenko and K.~Seip.
\newblock Large {G}{C}{D} sums and extreme values of the {R}iemann zeta
  function.
\newblock Preprint. Available at \url{http://arxiv.org/abs/1507.05840}.

\bibitem{bond}
A.~Bondarenko and K.~Seip.
\newblock G{CD} sums and complete sets of square-free numbers.
\newblock {\em Bull. Lond. Math. Soc.}, 47(1):29--41, 2015.

\bibitem{chand}
V.~Chandee and K.~Soundararajan.
\newblock Bounding {$\vert \zeta(\frac12+it)\vert $} on the {R}iemann
  hypothesis.
\newblock {\em Bull. Lond. Math. Soc.}, 43(2):243--250, 2011.

\bibitem{farmer}
D.~W. Farmer, S.~M. Gonek, and C.~P. Hughes.
\newblock The maximum size of {$L$}-functions.
\newblock {\em J. Reine Angew. Math.}, 609:215--236, 2007.

\bibitem{gal}
I.~S. G{\'a}l.
\newblock A theorem concerning {D}iophantine approximations.
\newblock {\em Nieuw Arch. Wiskunde (2)}, 23:13--38, 1949.

\bibitem{gs}
A.~Granville and K.~Soundararajan.
\newblock Extreme values of {$\vert \zeta(1+it)\vert $}.
\newblock In {\em The {R}iemann zeta function and related themes: papers in
  honour of {P}rofessor {K}. {R}amachandra}, volume~2 of {\em Ramanujan Math.
  Soc. Lect. Notes Ser.}, pages 65--80. Ramanujan Math. Soc., Mysore, 2006.

\bibitem{hed}
H.~Hedenmalm, P.~Lindqvist, and K.~Seip.
\newblock A {H}ilbert space of {D}irichlet series and systems of dilated
  functions in {$L^2(0,1)$}.
\newblock {\em Duke Math. J.}, 86(1):1--37, 1997.

\bibitem{hilber}
T.~Hilberdink.
\newblock An arithmetical mapping and applications to {$\Omega$}-results for
  the {R}iemann zeta function.
\newblock {\em Acta Arith.}, 139(4):341--367, 2009.

\bibitem{kotnik}
T.~Kotnik.
\newblock Computational estimation of the order of {$\zeta(\frac 12+it)$}.
\newblock {\em Math. Comp.}, 73(246):949--956 (electronic), 2004.

\bibitem{lamz}
Y.~Lamzouri.
\newblock On the distribution of extreme values of zeta and {$L$}-functions in
  the strip {$\frac12<\sigma<1$}.
\newblock {\em Int. Math. Res. Not. IMRN}, (23):5449--5503, 2011.

\bibitem{lewko}
M.~Lewko and M.~Radziwi{\l}{\l}.
\newblock Refinements of {G}\'al's theorem and applications.
\newblock Preprint. Available at \url{http://arxiv.org/abs/1408.2334}.

\bibitem{mont}
H.~L. Montgomery.
\newblock Extreme values of the {R}iemann zeta function.
\newblock {\em Comment. Math. Helv.}, 52(4):511--518, 1977.

\bibitem{robbins}
H.~Robbins.
\newblock A remark on {S}tirling's formula.
\newblock {\em Amer. Math. Monthly}, 62:26--29, 1955.

\bibitem{sound}
K.~Soundararajan.
\newblock Extreme values of zeta and {$L$}-functions.
\newblock {\em Math. Ann.}, 342(2):467--486, 2008.

\bibitem{titch}
E.~C. Titchmarsh.
\newblock {\em The theory of the {R}iemann zeta-function}.
\newblock The Clarendon Press Oxford University Press, New York, second
  edition, 1986.

\bibitem{voronin}
S.~M. Voronin.
\newblock Lower bounds in {R}iemann zeta-function theory.
\newblock {\em Izv. Akad. Nauk SSSR Ser. Mat.}, 52(4):882--892, 896, 1988.

\end{thebibliography}

\end{document}